\documentclass[12pt]{amsart}
\usepackage{epsfig}
\usepackage{amssymb, amsmath}
\usepackage{amsfonts,graphics,epsfig}
\usepackage{mathrsfs}
\usepackage{hyperref}
\usepackage[capitalise]{cleveref}
\usepackage{pb-diagram, pb-xy}
\usepackage[all]{xy}
\usepackage{color}
\usepackage{tikz}
\usepackage{tikz-cd}
\usepackage{setspace}
\usepackage[left=1.5in,right=1.5in,top=1.5in,bottom=1.5in]{geometry}
\usepackage{color}   
\usepackage{hyperref}
\hypersetup{
    colorlinks=true, 
    linktoc=all,     
    linkcolor=red,  
	citecolor=blue
}

\def\>{\geq}

\def\subset{\subseteq}

\newtheorem{theorem}{Theorem}[section]
\newtheorem{lemma}[theorem]{Lemma}

\newtheorem{corollary}[theorem]{Corollary}

\theoremstyle{remark}
\newtheorem{remark}[theorem]{Remark}
\theoremstyle{definition}
\newtheorem{definition}[theorem]{Definition}
\theoremstyle{definition}

\numberwithin{equation}{section}

\def\div{\operatorname{div}}

\def\Supp{\operatorname{Supp}}
\def\dim{\operatorname{dim}}

\author{Shikha Bhutani}
\address{Michigan State University\\
Michigan State University\\
619 Red Cedar Rd, \\
East Lansing, MI 48824}
\email{bhutanishikha11@gmail.com, bhutani4@msu.edu}

\date{}

\usepackage{fancyhdr}

\newcommand\shortitle{On K-V Vanishing for Surfaces of del Pezzo type over imperfect fields}

\begin{document}

\title{On Kawamata-Viehweg Vanishing for Surfaces of del Pezzo type over imperfect fields}

\begin{abstract}
We prove the Kawamata-Viehweg vanishing theorem for surfaces of del Pezzo type over imperfect fields of characteristic \( p > 5 \). As a consequence, we deduce the Grauert--Riemenschneider vanishing theorem for excellent divisorial log terminal threefolds whose closed points have (possibly imperfect) residue fields of positive
characteristic \( p > 5 \). Finally, under this setup, we show that three-dimensional klt singularities are rational.

\end{abstract}
\maketitle
\tableofcontents

\section{Introduction}

 Vanishing theorems are foundational tools in the Log Minimal Model Program (LMMP). One of the key vanishing theorems in birational geometry is the Kodaira vanishing theorem, along with its generalization, the Kawamata-Viehweg vanishing theorem. Although proofs of these theorems exist in characteristic zero, it is well known that they can fail in every positive characteristic in high enough dimensions (see \cite{totaro_failure}).

Fano varieties and their log generalizations are expected to behave better in this setting, and it is conjectured that Kodaira-type vanishing theorems hold for Fano varieties of a given dimension \( n \), provided the characteristic is sufficiently large, say \( p(n) \) (see \cite[Open Questions]{totaro_failure}). This conjecture has been confirmed in dimension \( 2 \) over perfect fields by \cite{cascini_log_2017}. In \cite{arvidson_vanishing}, the authors further determine the optimal bound \( p(2) = 5 \) (again in the perfect field setting). 

In this paper, we show that the result continues to hold over imperfect fields.

\newpage
 
\begin{theorem}\label{T1}
Let \( X \) be a surface of del Pezzo type over a field \( k \) of characteristic \( p > 5 \) (possibly imperfect). Let \( D \) be a Weil divisor and suppose that there exists an effective \( \mathbb{Q} \)-divisor \( \Delta \) such that \( (X, \Delta) \) is a klt pair and \( D - (K_X + \Delta) \) is big and nef. Then
\[
H^i(X, \mathcal{O}_X(D)) = 0 \quad \text{for all } i > 0.
\]

\end{theorem}

In 1981, Elkik proved that Kawamata log-terminal (klt) singularities are rational in characteristic zero (\cite{MR621766}). The main technical ingredient in Elkik’s proof is the Kawamata--Viehweg vanishing theorem. Since vanishing theorems fail in general in positive characteristic, it is natural to ask whether the implication `klt implies rational' continues to hold in characteristic \( p > 0 \). The implication fails in every positive characteristic in high enough dimensions (see \cite{totaro2024terminal3foldscohenmacaulay}). However, if we fix the dimension, the implication is expected to hold after a high enough prime.

In \cite{hacon_witaszek_rationality}, the authors show that there exists a natural number \( p_0 \) such that any three-dimensional klt singularity defined over a perfect field of characteristic \( p > p_0 \) is rational. The number \( p_0 \) is chosen large enough so that the Kawamata--Viehweg vanishing theorem holds for surfaces of del Pezzo type; that is, \( p_0 = p(2) \) as described above.

Using our main theorem, we deduce that klt singularities on excellent threefolds are rational.

\begin{corollary}(cf. \cite[Theorem 1.1]{hacon_witaszek_rationality}, \cite[Theorem 3.1]{bernasconi_kollar_vanishing}) \label{C1}
Let \( (X, \Delta) \) be a three-dimensional excellent dlt pair that admits thrifty resolution where \( X \) is an excellent scheme with a dualizing complex whose closed points have residue fields of characteristics \( p > 5 \).  Then \((X,B)\) is a rational pair (see Definition \ref{DR}) for every \(B \subseteq \lfloor \Delta \rfloor\).

In particular, three-dimensional klt singularities under the same setup are rational and hence, Cohen-Macaulay.

 Moreover, if \( X \) is \( \mathbb{Q} \)-factorial and \( D \) is a Weil divisor on \( X \), then \( \mathcal{O}_X(D) \) is Cohen--Macaulay.
\end{corollary}

In Corollary \ref{C1} the assumption on the characteristics is optimal, as shown by
the examples in \cite{totaro2024terminal3foldscohenmacaulay}.

The idea is to consider a thrifty log resolution (For the definition of a thrifty resolution, see \cite[Definition 2.79]{kollar_singularities_2013}) of the pair \( (X, \Delta) \) and to show that it is a rational resolution. 
Thrifty resolutions are known to be rational in characteristic zero (see \cite[Corollary 2.86]{kollar_singularities_2013}). However, in the positive characteristic, this is not known, and the question is closely related to the failure of the Grauert--Riemenschneider vanishing theorem in characteristic \( p \) (see Definition \ref{D4}).

We show that Grauert--Riemenschneider vanishing holds for dlt threefolds as a consequence of our main theorem, and therefore Corollary~\ref{C1} follows. More details are provided in Section~\ref{S5}.

\begin{corollary}(cf. \cite[Theorem 1.3]{bernasconi_kollar_vanishing} \label{C2}
    Let $(X, \Delta)$ be a three-dimensional dlt pair, where $X$ is an excellent
scheme with a dualizing complex whose closed points have residue fields
of characteristic $p > 5$. Then G-R vanishing holds over $(X, \Delta)$.
\end{corollary}

The proof is the same as the proof of \cite[Theorem 3]{bernasconi_kollar_vanishing}. We replace the result of \cite{arvidson_vanishing} with Theorem \ref{T1}.

\begin{remark}
After completing this work, we learned that Jefferson Baudin, Tatsuro Kawakami, and Linus R\"{o}sler independently established that three-dimensional Cohen-Macaulay klt singularities satisfy Grauert--Riemenschneider vanishing, using different methods (see \cite{baudin2025grauertriemenschneidervanishingcohenmacaulayschemes}).
\end{remark}

As a direct consequence of Corollary \ref{C2}, we have the following. 

\begin{corollary}(cf. \cite[Theorem 5.1]{bernasconi_kollar_vanishing}) \label{C3}
Let $k$ be a field of characteristic $p> 5$.
Let $f \colon X \rightarrow Z$ be a projective contraction morphism between quasi-projective normal varieties over $k$.
Suppose that there exists an effective $\mathbb{Q}$-divisor $\Delta \geq 0$ such that
\begin{enumerate}
\item $(X, \Delta)$ is a klt threefold pair;
\item  $-(K_X+\Delta)$ is $f$-big and $f$-nef;
\item $\dim(Z) \geq 1$.
\end{enumerate} 
The natural map $\mathcal{O}_Z \rightarrow \mathbf{R}f_*\mathcal{O}_X $ is then an isomorphism.
\end{corollary}
\begin{remark}
    
The study of singularities plays a central role in birational geometry. Broadly speaking, there are two major approaches to studying singularities. One approach is to analyze their local properties; for instance, in Corollary \ref{C1}, we show that three-dimensional klt singularities are rational. 

Another perspective is to consider a proper variety and study the behavior of its singularities globally. In particular, given a variety \( X \), it is natural to examine the locus where \( X \) fails to be klt. This locus is denoted by \( \mathrm{Nklt}(X, B) \).

In \cite{Filipazzi_2024}, the authors prove that the Shokurov--Kollár connectedness principle holds for threefolds in both positive and mixed characteristic, under the assumption that all residue fields have characteristic \( p > 5 \). They also characterize the cases in which \( \mathrm{Nklt}(X, B) \) fails to be connected.

Moreover, as noted in \cite[Remark 1.3]{Filipazzi_2024}, \cite[Theorem 1.2]{Filipazzi_2024} continues to hold in the setting of imperfect base fields since we have the required vanishing theorems over the imperfect fields [Corollary \ref{C2}].
\end{remark}

\begin{proof}[Sketch of the proof of Theorem \ref{T1}]\label{P1}
We bring the entire setup to a smaller field \( E \), as in Lemma \ref{l1}. We observe that \( X_E \) is the generic fiber of a morphism \( \phi \colon \mathcal{X} \to \mathcal{V} \), where $\mathcal{X} \text{ and } \mathcal{V}$ are varieties over perfect field and the function field of \( \mathcal{V} \) is \( E \). We proceed by induction on the dimension of \( \mathcal{V} \). The base case \( \dim \mathcal{V} = 0 \) follows from \cite{arvidson_vanishing}.

Assume that Kawamata-Viehweg vanishing holds over the generic fiber of \( \phi \) when \( \dim \mathcal{V} = n - 1 \). We aim to show the vanishing when \( \dim \mathcal{V} = n \). The idea is to localize at the generic point \( \eta \) of a general hyperplane section of \( \mathcal{V} \), and prove that
\[
R^i \phi_{\eta *} \mathcal{O}_{\mathcal{X}_\eta}(D_\eta) = 0.
\]

We make a birational transformation of \( \mathcal{X}_\eta \) to obtain a pair \( ({Y}, \Delta_{{Y}}) \), where \( {Y} \) is a \( \mathbb{Q} \)-factorial plt threefold and \( S = \lfloor \Delta_Y \rfloor \) is a surface of Fano type.
We show that it is sufficient to prove 
\[
R^i g_* \mathcal{O}_{{Y}}(D_{{Y}}) = 0.
\] where $g: {Y} \rightarrow \mathcal{V}_{\eta}$ is the birational transform. 
We apply the induction hypothesis to \( S \) to show that Kawamata--Viehweg vanishing holds for \( g|_S \), and then use \cite[Proposition 7]{bernasconi_kollar_vanishing} to deduce the vanishing
\[
R^i g_* \mathcal{O}_{{Y}}(D_{{Y}}) = 0
\]
in a neighborhood of \( g(S) \).

\end{proof}

\textbf{Acknowledgement.}
 The author is immensely grateful to Joe Waldron, her advisor, for suggesting this problem and for his unwavering support, insightful advice, fruitful discussions, and infinite patience. This work would be incomplete without his guidance. We would
like to thank Fabio Bernasconi and Karl Schwede for reading this manuscript and providing many
valuable feedback and helpful comments.
This work was supported by NSF Grant DMS \#2401279 and by the Simons Foundation (award \#850684, JW).

 \section{Preliminaries}
\subsection{Definitions}
In this article, we work over an arbitrary field \( k \) (possibly imperfect) of characteristic \( p > 0 \). We say that \( X \) is a variety over \( k \), or a \( k \)\textit{-variety}, if \( X \) is an integral scheme that is separated and of finite type over \( k \). A surface is a \( k \)-variety of dimension \( 2 \).

\begin{definition}\label{D1}
   We call a pair \( (X, \Delta) \) a \textit{log pair} if \( X \) is a normal \( k \)-variety and \( \Delta \) is a boundary divisor; that is, \( \Delta \) is an effective \( \mathbb{Q} \)-divisor with coefficients in the interval \( [0,1] \), such that \( K_X + \Delta \) is \( \mathbb{Q} \)-Cartier.

\end{definition}
We refer \cite{kollar_birational_1998} for basic definitions in birational geometry and singularities. We say that a variety is \textit{ $\mathbb{Q}$-factorial} if every Weil divisor is $\mathbb{Q}$-Cartier.
We recall that klt surface pairs are $\mathbb{Q}$-factorial.
\begin{definition}\label{D2}
Let $k$ be a field. We say that normal $k$-projective surface $X$  is a  \textit{surface of del Pezzo type} if there exists a boundary divisor $\Delta$ such that the following holds:
\begin{enumerate}
    \item $(X, \Delta) $ is a klt pair
    \item $-(K_X+ \Delta)$ is ample
\end{enumerate}
The pair $(X, \Delta)$ is called a \textit{log del Pezzo pair}.
\end{definition}
We recall the definition of rational singularities for pairs, following \cite[Definition 2.80]{kollar_singularities_1996}.

\begin{definition}[Rational Pair] \label{DR}
    Let $X$ be a normal variety and $D \subset X$ be a reduced divisor.
Let $\phi : Y \rightarrow X$ be a resolution such that
$(Y,D_Y := \phi^{-1}_* (D))$ is snc. Then
$\phi : (Y, D_Y) \rightarrow (X, D)$ is called a \textit{rational resolution} if
\begin{enumerate}
    \item $\mathcal{O}_X (-D) \sim \phi_* \mathcal{O}_Y (-D_Y)$
    \item $R^i \phi_* \mathcal{O}_Y (-D_Y) =0$ for $i>0$
    \item $R^i \phi_* \omega_Y (D_Y) =0$ for $i>0.$
\end{enumerate}
A pair $(X,D)$ is called \textit{rational} if it has a thrifty rational resolution.
\end{definition}

In the proof of Corollary \ref{C1}, we show that for every reduced divisor $ B \subset \lfloor {\Delta} \rfloor$, the pair $(X, B)$ is a rational pair. We show that every thrifty log resolution of $(X, B)$ is rational, and the main tool of the proof is G-R vanishing, which is defined as follows:

\begin{definition}\label{D4}
Let $(X, \Delta)$ be a pair,
where $X$ is a normal, excellent
scheme with a dualizing complex and $\Delta$ is a boundary on $X$. 
We assume that the log resolutions of singularities exist for the pairs in G-R vanishing-type statements.

We say that (strong) \textit{G-R vanishing} holds over 
$(X, \Delta)$ if the following is satisfied
for every log resolution $g : (X', E + g^{-1}_* \Delta) \longrightarrow (X, \Delta)$,
where $E$ is the exceptional
divisor of $g$.

Let $D'$ be a Weil $\mathbb{Z}$-divisor and $\Delta'$ an effective $\mathbb{R}$-divisor on $X'$. 
Assume
that $g_* \Delta' \leq \Delta, \lfloor \text{Ex} ( \Delta')\rfloor = 0$ (where Ex$(\Delta'
)$ denotes the $g$-exceptional
part of $\Delta'$
), and
\begin{enumerate}
    \item (for G-R Vanishing) $D' \sim_{g,\mathbb{R}} K_{X'} + \Delta'$,
    \item (for strong G-R Vanishing) $D' \sim_{g,\mathbb{R}} K_{X'} + \Delta' + (g$-nef $\mathbb{R}$-divisor).
\end{enumerate}
Then $R^i g_* \mathcal{O}_{X'} (D') = 0$ for $i > 0$.
\end{definition}
\subsection{Related Results}
In this section, we summarize the results already known for del Pezzo-type surfaces, which we will need later. We begin by stating a theorem of \cite[Theorem 1.1]{arvidson_vanishing}, which will be generalized in this work.

\begin{theorem}[{\cite[ Theorem 1.1]{arvidson_vanishing}}]\label{ABL19}
	Let $X$ be a surface of del Pezzo type over a perfect field $k$ of characteristic $p>5$.
	Let $D$ be a Weil divisor on $X$ and suppose that there exists an effective $\mathbb{Q}$-divisor $\Delta$ such that $(X, \Delta )$ is a klt pair and $D-(K_X+\Delta)$ is big and nef. 
	Then 
	$$H^i(X, \mathcal{O}_X(D))=0 \text{ for all } i>0. $$
    
\end{theorem}

The assumption of a perfect field is necessary for this result because it uses Lacini's Ph.D. thesis \cite{lacini_dP} to deal with the case
of klt del Pezzo surfaces of Picard rank one. When generalizing to an arbitrary field, we require the Minimal Model Program for arithmetic threefolds whose residue characteristics $p>5$.
\begin{theorem}\cite[Theorem G]{bhatt2022globallyregularvarietiesminimal}
  Let $(X,B)$ be a three-dimensional
$\mathbb{Q}$-factorial dlt pair with $\mathbb{R}$- boundary, which is projective over $T$. Assume that the image of $X$ in $T$ is
of positive dimension and that $T$ has no residue fields of characteristic $2, 3$ or $5.$
If $K_X + B$ is pseudo-effective, then we can run a $(K_X + B)-$ MMP and any sequence of steps
of this MMP terminates with a log minimal model.
If $K_X + B$ is not pseudo-effective, then we can run a $(K_X + B)-$ MMP with scaling over $T$
which terminates with a Mori fiber space.
\end{theorem}

\section{ Lemmas and Propositions}

 In this section, we prove several lemmas and propositions needed for the proof of the main theorem and its corollaries. Some of these results are known in more restricted settings, and we generalize them to the context required for our arguments. \\
 As outlined earlier, we reduce the entire setup to a smaller field, and the following lemma facilitates this reduction.

\begin{lemma}\label{l1}
Let $X$ be a surface of del Pezzo type over a field $k$ of characteristic $p>5$ such that $(X,\Delta)$ is a klt pair. Let $D$ be a Weil divisor on $X$ so that $D-(K_X +\Delta)$ is big and nef. 

Then there exists a finitely generated extension $E$ of $\mathbb{F}_p $ such that $(X_E,\Delta_E)$ and $D_E$ satisfy the same setup where $(X, \Delta)= (X_E, \Delta_{E}) \otimes_E k$ and $D_E$ is the push forward of $D$ under $ i: X \longrightarrow X_E$.
\end{lemma}

\begin{proof}
Let $E$ be a finite transcendental extension of $\mathbb{F}_p $ which contains the coefficients of the equations of some quasi-projective embeddings of $X$, every
component of $D$, $K_X$, $\Delta$.
First, we intend to show that $X_E$ is a surface of del Pezzo type.
Given that $X$ is a surface of del Pezzo type, there exists an effective $\mathbb{Q}$-divisor $\Delta_1$ such that $(X,\Delta_1)$ is klt pair and $-(K_X+\Delta_1)$ is ample. 
Let $$\phi_{-(K_X+\Delta_1)}: X \rightarrow \mathbb{P}^m$$ be the morphism induced by a high enough multiple of $-(K_X+\Delta_1).$ We extend $E$ so that it contains the equation of embeddings of $X$ by $\phi_{-(K_X+\Delta_1)}$, $K_X$ and $\Delta_1$. 
Then this embedding descends to $E$, and we get, $$\phi_{-(K_{X_{E}}+\Delta_{1,E})}: X_E \rightarrow \mathbb{P}^m_E$$
This shows that $X_E$ is a normal $E$-projective surface with $-(K_{X_E} +\Delta_{1,E})$ an ample divisor. 
Note that $(X, \Delta_1)= (X_E, \Delta_{1,E}) \otimes_E k.$
Next, we intend to show that $(X_E, \Delta_{1,E})$ is a klt pair. 
Since $(X, \Delta_1)$ is a klt pair, there exists a log resolution 
$$f: Y \rightarrow X$$ such that if we write $$f^*(K_{X} +\Delta_1) = K_Y +C$$ then the coefficient $c_j$ of $C_j$ in $C$ lies in $(-\infty,1).$ We extend the field $E$ to contain the coefficients of the equation of log resolution $f$ and equations of $C_j's$.
This assures the existence of a log resolution such that the required condition on the coefficient of the exceptional divisor holds. And we conclude that $(X_E, \Delta_{1,E})$ is a klt pair and therefore a surface of del Pezzo type over $E$.

By the same argument as above, we can say that $(X_E, \Delta_E)$ is a klt pair. Next, we show that the divisor $D_E-(K_{X_E} +\Delta_E)$ is nef and big. It is enough to produce an effective divisor $N_E$ such that $D_E-(K_{X_E} +\Delta_E)+\frac{1}{m} N_E$ is ample for $m \gg0.$ But $D-(K_{X} +\Delta)$ is nef and big, so there exists an effective divisor $N$ such that $D-(K_{X} +\Delta)+\frac{1}{m} N$ is ample for $m \gg 0.$ We can extend $E$, so that it contains the equations of $N$.  And we conclude that $D_E-(K_{X_E} +\Delta_E)$ is nef and big. Note that all the extensions of $E$ in the proofs are finite. So $E$ is a finitely generated extension of $\mathbb{F}_p.$ This completes the proof.

\end{proof}

Next, we present a series of results that appear in less general form in the literature; we also adapt some of them to suit the requirements of the Main Theorem. We start by generalizing \cite[Proposition 2.15]{gongyo_rational_2015-1} in the imperfect field setting.

\begin{lemma}\label{l2}(cf. \cite[Proposition 2.15]{gongyo_rational_2015-1})
    Let $k$ be an F-finite field of char $p>5.$ Let $\phi:X \rightarrow V$ be a projective morphism of normal quasi-projective varieties over $k$ with the following properties:
    \begin{itemize}
        \item $(X, \Delta)$ is a klt threefold.
        \item $0< \dim V \leq 3$
        \item $\phi_* \mathcal{O}_X = \mathcal{O}_V$ and $-(K_X +\Delta)$ is $\phi$-nef and $\phi$-big. 
    \end{itemize}
    Fix a closed point $v \in V$. Then there exists a commutative diagram of quasi-projective normal varieties.  
 \begin{center} 
\begin{tikzcd}
W \arrow[d, "\varphi"'] \arrow[r, "\psi"] & Y \arrow[d, "g"] \\
X \arrow[r, "\phi"']                      & V               
\end{tikzcd}
\end{center}
and an effective divisor $\mathbb{Q}$-divisor $\Delta_Y$ on $Y$ which satisfy the following properties:
\begin{enumerate}
    \item $(Y, \Delta_Y)$ is a $\mathbb{Q}$-factorial plt threefold and $-(K_Y + \Delta_Y)$ is $g$-ample.
    \item $-\left \lfloor{\Delta_Y}\right \rfloor $ is $g$-nef and $g^{-1}(v)_{\text{red}} = \left \lfloor{\Delta_Y}\right \rfloor  $
    \item $W$ is a regular $3$-fold and both $\varphi$ and $\psi$ are projective birational morphism.
\end{enumerate}
In particular, $\left \lfloor{\Delta_Y}\right \rfloor $ is a surface of del Pezzo type.
\end{lemma}

    \begin{proof}
      The proof works on the same lines as \cite[Proposition 2.15]{gongyo_rational_2015-1}. We highlight the central part of the proof that is useful to us. 
      The idea is to consider the log resolution $h : Z \longrightarrow X$ of $X$ and choose a Cartier divisor $D_V$ on $V$ such that $v \in \text{Supp } D_V$ and, $$0 \sim_V h^* \phi^* D_V =M+F$$ where $M$ is base point free and $F= \sum f_i E_i$ is an effective divisor whose support is the same as the fiber of $v$ under the map $Z \longrightarrow V.$ Then consider $$
K_Z+h_*^{-1}\Delta=h^*(K_X+\Delta)+\sum a_iE_i
$$ and write $-h^*(K_X+\Delta) =A+G$ where $A$ is an ample $\mathbb{Q}$-divisor over $V$ 
and $G$ is an effective $h$-exceptional $\mathbb{Q}$-divisor. Let $G=\sum_{i}g_iE_i$ and choose $\epsilon$ a small positive rational number such that $a_i-\epsilon g_i >-1$ for any $i$. The authors then proceed to write 
\[
K_Z+h_*^{-1}\Delta+\epsilon A-(1-\epsilon)h^*(K_X+\Delta)=\sum (a_i-\epsilon g_i) E_i.
\]
Since $M+F \sim_V 0$, they find $\lambda \in \mathbb{Q}_{>0}$ with
\[
K_Z+h_*^{-1}\Delta+ \big( \epsilon A- (1-\epsilon) h^*(K_X+\Delta)+\lambda M \big) \sim_{\mathbb{Q}, V} \sum (a_i-\epsilon g_i-\lambda f_i)E_i,
\]
such that $a_i-\epsilon g_i-\lambda f_i \geq -1$ holds for any $i$ and at least one index attains $-1$. 
They let $S_Z$ to be a prime divisor which attains $-1$ and write $b_i :=(a_i-\epsilon g_i-\lambda f_i)+1$. Clearly, $b_i \geq 0$. It is important to note that $S_Z$ and $\sum_{i \in I'} b_iE_i$ may not be $h$-exceptional, but $\Supp \sum_{i \in I'} b_iE_i$ contains all the $h$-exceptional prime divisors 
whose images in $X$ are not contained in $X_v$ (In particular, it contains all the exceptional divisors whose centers lie in the generic fiber of $\phi :X \longrightarrow V$, see Remark \ref{R1}). Finally, they run MMP three times: 
\begin{itemize}
    \item $(K_Z+B_Z+S_Z+\sum_{i \in I'} E_i)$-MMP over $X$ where $B_Z := h_*^{-1}\Delta+ \big( \epsilon A- (1-\epsilon) h^*(K_X+\Delta)+\lambda M \big)$ (replace $\big( \epsilon A- (1-\epsilon) h^*(K_X+\Delta)+\lambda M \big)$ by an ample $\mathbb{Q}$-divisor $A'$ if needed).
    \item $(K_{Z'}+B_{Z'}+S_{Z'}+\sum _{i \in I'} E_i')$-MMP over $V$
    \item $(K_{Z''}+B_{Z''}+ \frac{1}{2}S_{Z''}+\sum _{i \in I'} E_i'')$-MMP over $V$ again.
\end{itemize}
We replace Lemma 2.2, Lemma 2.6, Lemma 2.9, and Lemma 2.10 by Theorem 1.1, Theorem 6.6,  Theorem 1.4, and Theorem 1.3, respectively, in \cite{das_waldron_imperfect} to shift from perfect fields to an imperfect field setting. It is noted that the end result 
$(Y, B_{Y} + \frac{1}{2}S+\sum _{i \in I'} E_{i, Y})$ is a minimal model equipped with $g:Y \to V$ and all the prime divisors in $\sum _{i \in I'} E_i$ (which does not contain $S_Z$) are contracted in this MMP. 
Refer to \cite[Proposition 2.15]{gongyo_rational_2015-1} for more details.
    \end{proof}

\begin{remark}\label{R1}
 In the proof of Theorem \ref{T1}, we will need to show that when we restrict $\psi \circ \varphi^{-1}$ to the generic fiber $X_E$ of $\phi: \mathcal{X} \longrightarrow \mathcal{V}$, it turns out to be a morphism.  For this we only need to ensure that the exceptional divisors of the resolution of singularities whose centers lie in $X_E$ are contracted under the MMP-map. We see in the proof above that the curves contracted by the first MMP are contained in
$$\bigcup _{i \in I'} \Supp E_i$$ since it is also $(\sum_{i \in I'} b_iE_i)$-MMP. The MMP terminates with a minimal model 
$(Z', B_{Z'} + S_{Z'}+\sum_{i \in I'} E_i')$ over $X$. It is observed that  $E_i'=0$ if $E_i$ is contracted and $\Supp(\sum b_iE_i') \subset Z'_v$. Since $v$ is a closed point in $V$, we see that all the exceptional divisors whose centers lie in the generic fiber have been contracted.
 
\end{remark}
We observe that a three-dimensional plt pair, under suitable hypotheses, is purely globally $+$-regular (see \cite[Corollary~7.15]{bhatt2022globallyregularvarietiesminimal}). For the scheme $Y$ constructed in Lemma~\ref{l2}, we show that certain special divisorial sheaves are Cohen–Macaulay.

\begin{lemma}\label{l3}(cf. \cite[Theorem 2]{kollar_local_2011}, \cite[Theorem 3.1]{patakfalvi_depth_2014})
        Suppose that $(R,m)$ is local and that $(X= \operatorname{Spec} R, \Delta)$
is $+$-regular $\mathbb{Q} -$factorial pair. Furthermore, suppose that $0 \leq \Delta' \leq \Delta $ is such that $r D \sim r \Delta'$
for some integral divisor $D$ and some integer $r>0$. Then $\mathcal{O}_X (-D)$ is Cohen-Macaulay. 

    \end{lemma}

\begin{proof}
By \cite[Proposition 3.5.4]{bruns_cohen-macaulay_1993}, it suffices to show that  $
H^i_{\mathfrak{m}}(X, \mathcal{O}_X(D)) = 0$ for all  $i < \dim X.$ Consider the integral closure of the completion of \((R, \mathfrak{m})\),
\[
R \xrightarrow{\phi} \widehat{R} \xrightarrow{\varphi} \widehat{R}^+.
\]

This induces a morphism of the affine schemes $$\operatorname{Spec} \widehat{R}^+ \xrightarrow{f} \operatorname{Spec} \widehat{R} \xrightarrow{g} \operatorname{Spec} {R}$$ and therefore a map on the structure sheaves
\[
F: \mathcal{O}_{\widehat{X}} \longrightarrow f_*\mathcal{O}_{\widehat{X}^+} \text{  and  } G: \mathcal{O}_X \longrightarrow g_*\mathcal{O}_{\widehat{X}}.
\]

Since $\widehat{R}$ is $+$-regular, the composition map \(\varphi: \widehat{R} \to \widehat{R}^+\) is pure, being a filtered colimit of split maps. Since \(\widehat{R}\) is complete and local, the map \(\widehat{R} \to \widehat{R}^+\) splits. Similarly, we see that \[
 \mathcal{O}_{\widehat{X}} \longrightarrow f_*\mathcal{O}_{\widehat{X}^+}(\Delta')
\] 
splits.

Twisting by \(-D\), reflexifying, we get:
 \[ \mathcal{O}_{\widehat{X}}(-D)\longrightarrow f_*\mathcal{O}_{\widehat{X}^+}(\Delta'-D)
\] 
The above map splits because $-D$ is reflexive of rank one.
We take the local cohomology to obtain a map
\[
H^i_{\mathfrak{m}}(\widehat{X}, \mathcal{O}_{\widehat{X}}(-D)) \longrightarrow H^i_{\mathfrak{m}}(\widehat{X}^+, \mathcal{O}_{\widehat{X}^+}(\Delta'-D)).
\]
This map is injective because \( \mathcal{O}_{\widehat{X}}(-D)\longrightarrow f_*\mathcal{O}_{\widehat{X}^+}(\Delta'-D)
\) splits. 
Since $r \Delta' \sim r D$, there exists a finite cover of $X$ where  $ \Delta' \sim  D.$ Indeed, let $ r(\Delta' -D)= \div f$, then $X_1 =\operatorname{Spec} \frac{R[T]}{(T^r -f)}$ is the finite cover where  $ \Delta' \sim  D$. We conclude that  $ \Delta' \sim  D$ on the absolute integral closure $X^+ =\operatorname{Spec} R^+$ and the cohomology on the right hand side coincide with that of $R^+$. By Bhatt vanishing theorem \cite[Theorem 5.1]{bhatt_cohen}, the right-hand side vanishes for all \(i < \dim X\), and thus
\[
H^i_{\mathfrak{m}}(\widehat{X}, \mathcal{O}_{\widehat{X}}(-D)) = 0.
\]

Next, note that the completion map \(g: R \to \widehat{R}\) is faithfully flat. Consequently, the induced morphism on structure sheaves,
\[
G: \mathcal{O}_X \longrightarrow g_*\mathcal{O}_{\widehat{X}},
\]
is faithfully flat as well. Twisting \(G\) by \(\mathcal{O}_X(-D)\), and noting that \(\mathcal{O}_X(-D)\) is reflexive of rank one, we obtain a faithfully flat morphism
\[
\mathcal{O}_X(-D) \longrightarrow g_*\mathcal{O}_{\widehat{X}}(g^*(-D)),
\]
which induces an isomorphism on local cohomology. Thus,
\[
H^i_{\mathfrak{m}}(X, \mathcal{O}_X(-D)) = 0.
\]

\end{proof}
We use Lemma~\ref{l3} to establish the short exact sequence from \cite[Proposition~3.1]{hacon_witaszek_rationality} in the mixed characteristic setting.

\begin{lemma}\label{l4} (cf. \cite[Proposition 3.1]{hacon_witaszek_rationality})
Let $(Y, S + \Delta_Y)$ be a three-dimensional $\mathbb{Q}$-factorial plt pair, where $Y = \operatorname{Spec} R$ for an excellent ring $R$ admitting a dualizing complex, and such that the residue fields of closed points of $R$ have characteristic $p > 5$. Then for any Weil divisor $D_Y$ and every integer $m \geq 1$, there exists an exact sequence
\[
0 \longrightarrow \mathcal{O}_Y(D_Y - (m+1)S) \longrightarrow \mathcal{O}_Y(D_Y - mS) \longrightarrow \mathcal{O}_S(G_m) \longrightarrow 0,
\]
where 
\[
G_m \sim_{\mathbb{R}} -m S|_S + D_Y|_S - \Delta_m
\]
for some divisor $0 \leq \Delta_m \leq \operatorname{Diff}_{\bar{S}}(\Delta_Y').$

\end{lemma}
\begin{proof}
It is sufficient to prove the exactness of the sequence in a neighborhood of $S$. 
So we assume that $Y$ is spectrum of a local ring and therefore, globally $+$-regular and $S$ 
is normal by \cite[Corollary 7.15]{bhatt2022globallyregularvarietiesminimal}.  We consider the short exact sequence: $$0 \rightarrow \mathcal{O}_Y (K_Y+D_Y-(m+1)S) \rightarrow \mathcal{O}_Y (K_Y+D_Y-mS) \rightarrow i_* \varepsilon \rightarrow 0$$ where $\varepsilon$ is a sheaf supported on $S$ and $i: S \rightarrow X$ is the inclusion.
The sheaf $\mathcal{O}_Y (K_Y+D_Y-(m+1)S)$ and $\mathcal{O}_Y (K_Y+D_Y-mS)$ are reflexive of rank one, so it is $S_3$. 
Indeed, let $D$ be a reflexive sheaf of rank $1$. Since $X$ is $\mathbb{Q}$-factorial, there exists an $m \in \mathbb{N}$ such that $mD$ is Cartier. Since $X$ is local, we see that $mD$ is trivial. In particular, $-mD \sim m0$. We apply Lemma \ref{l3} to conclude $\mathcal{O}(D)$ is Cohen-Macaulay, in particular, it is $S_3$. Then, the proof of \cite[Proposition 3.1]{hacon_witaszek_rationality} carries over.
\end{proof}   

Using Lemmas~\ref{l3} and~\ref{l4}, we establish a special case of \cite[Proposition~2.3]{bernasconi_kollar_vanishing} that is suitable for our main proof.

\begin{lemma}\label{l5} \cite[Proposition 2.3]{bernasconi_kollar_vanishing}
Let \((Y, S + \Delta_Y)\) be a three-dimensional $\mathbb{Q}-$ factorial excellent plt pair, \(g \colon Y \to V\) a proper morphism, and \(D_Y\) a Weil \(\mathbb{Z}\)-divisor on \(Y\). Assume that:
\begin{enumerate}
    \item \(D_Y \sim_{g,\mathbb{R}} K_Y + S + \Delta' + L\) for some \(0 \leq \Delta' \leq \Delta_Y\),
    \item \(-S\) is \(g\)-nef,
    \item \(L\) is a $g$-nef and $g$-big and K-V vanishing holds for $g|_{\overline{S}}$.
\end{enumerate}
Then \(R^i g_* \mathcal{O}_Y(D) = 0\) for all \(i > 0\) near \(g(S)\).
\end{lemma} 
\begin{proof}
    The proof follows exactly as in \cite[Proposition~2.3]{bernasconi_kollar_vanishing}. We use Lemma~\ref{l4} to establish the exactness of the short exact sequence
\[
0 \longrightarrow \mathcal{O}_Y(D_Y - (m+1)S) 
\longrightarrow \mathcal{O}_Y(D_Y - mS) 
\longrightarrow \mathcal{O}_S(G_m) 
\longrightarrow 0,
\]
where
\[
G_m \sim_{\mathbb{R}} -m S|_S + D_Y|_S - \Delta_m.
\]
At the end of the proof, the term $R^i g_* v_* \mathcal{O}_{S}(G_m)$ vanishes since $(S, \operatorname{Diff}_S(\Delta))$ is a klt pair (by Subadjunction Formula). The remainder of the proof proceeds unchanged.

\end{proof}

\section{ Main Theorem}\label{S5}
    \begin{proof}[Proof of Theorem \ref{T1}]
 Let $(X, \Delta)$ be a klt pair over $k$ that satisfies the hypothesis of Theorem \ref{T1}.
 Let $Z$ be the non-SNC locus of $X$. 
 We consider $\pi : Y \rightarrow X$ a log resolution of $X$ so that every Weil divisor over $Z$ has positive log discrepancy. 
 Let $E$ be a finitely generated extension of $\mathbb{F}_p$ which contains the coefficients of equations of quasi-projective embeddings of $X$, components of $ D,\Delta, Z,\pi,K_X $ as in Lemma \ref{l1}. 
 Note that this is a finitely generated extension of $\mathbb{F}_p$. We see that $(X_E, \Delta_E)$ satisfy the same setup by Lemma \ref{l1}.
 It is sufficient to show the vanishing for the pair  $(X_E, \Delta_E)$. Indeed, in the diagram:\\
 \[\begin{tikzcd}
	X & {X_E} \\
	{\text{Spec } \space k} & {\text{Spec } \space E}
	\arrow[from=1-1, to=1-2]
	\arrow[from=1-1, to=2-1]
	\arrow[from=1-2, to=2-2]
	\arrow[from=2-1, to=2-2]
\end{tikzcd}\]\\
We have $H^i(X_E, D_E) \otimes_E k = H^i(X,D)$.
 We see that $X_E$ forms a generic fiber of $\phi: \mathcal{X} \rightarrow \mathcal{V}$ where $\mathcal{X}$ and $ \mathcal{V}$ are varieties over perfect fields and $\mathcal{V}$ has function field $E$. 
 So, we shall show the vanishing for the generic fiber of $\phi$.  The statement is local in the generic fiber of $\phi$, we are free to shrink $Z$ as required. We will assume that ${\phi}_* \mathcal{O_X}= \mathcal{O_V} $ by the Stein factorization of $\phi$. We proceed by induction on the dimension of $\mathcal{V}$.\\
 \underline{\textbf{Base Case}}: $\dim \mathcal{V} =0$ \\
 The result follows from \cite[Theorem 1.1]{arvidson_vanishing}.\\
 \underline{\textbf{Inductive hypothesis}}: $\dim \mathcal{V} \leq n-1$\\
 We assume that Kawamata-Viehweg vanishing holds over the generic fiber of $\phi$ when $\dim \mathcal{V} \leq n-1$.\\ 
 \underline{\textbf{Inductive step}}: $\dim \mathcal{V} =n$\\
We claim that the general fibres of $\phi$ are integral. By  \cite[\href{https://stacks.math.columbia.edu/tag/0553}{Section 0553}]{stacks-project} and \cite[\href{https://stacks.math.columbia.edu/tag/0574}{Section 0574}]{stacks-project}, it suffices to show that the generic fibre of $\phi$, $X_E$, is geometrically integral. But it follows from  \cite[Corollary~5.5]{bernasconi_tanaka}, since $X_E$ is a surface of del Pezzo type over $k$. 

 Let $H$ be a general hyperplane on $ \mathcal{V}$, we write $H'=\phi^* H$. We can
choose $H$ so that $H'$ is integral over the generic point of $H$. Furthermore,
if it is not integral everywhere, it must fail in some proper closed subset of $\mathcal{V}$, so we can
assume that F is integral by shrinking $\mathcal{V}$. Let $\eta$ be the generic point of $H$, and $\mathcal{X}_{\eta} =\mathcal{X} \otimes \mathcal{V}_{\eta}$. 
We see that $X_E $ forms the generic fiber of the morphism $\phi_{\eta}: \mathcal{X}_{\eta} \rightarrow \mathcal{V}_{\eta}$. 
So, it is sufficient to show $R^i \phi_{\eta *} \mathcal{O}_{\mathcal{X}_{\eta}} (\mathcal{D}_{\eta})=0$ instead.   
Note that $\mathcal{X}_{\eta}$ is of dimension $3$, projective over excellent base $\mathcal{V}_{\eta}$. 
We apply the Lemma \ref{l2} to the projective morphism $\phi_{\eta}: \mathcal{X}_{\eta} \rightarrow \mathcal{V}_{\eta} $ between normal quasi-projective varieties to get the commutative diagram as follows:\\
 \begin{center}
  \begin{tikzcd}
W \arrow[d, "\varphi"'] \arrow[r, "\psi"] & Y \arrow[d, "g"] \\
\mathcal{X}_{\eta} \arrow[r, "\phi_{\eta}"']                      & \mathcal{V}_{\eta}               
\end{tikzcd}
\end{center}
    where $\varphi$ and $\psi $ are the projective birational morphism, $Y$ comes from the construction in Lemma \ref{l2}, $(Y, \Delta_Y)$ is $\mathbb{Q}-$ factorial plt threefold, $-(K_Y +\Delta_Y) $ is $g$-ample, $-\lfloor \Delta_Y \rfloor$ is $g$-nef and $g^{-1}(v)_{\text{red}} =\lfloor \Delta_Y \rfloor$ for a closed point $v \in \mathcal{V}_{\eta}$. We write  $S=\lfloor \Delta_Y \rfloor$ and  claim that it is sufficient to prove \[
R^i g_* \mathcal{O}_Y(D_Y) = 0 \quad \text{for all } i > 0
\] near $g(S).$ Indeed, let $F := \psi \circ \varphi^{-1}$,  then the restriction of $F$ to $X_E$, say, $f := F|_{X_E}$, is a morphism (See Remark \ref{R1}).
Therefore, $f$ satisfies the relative Kawamata--Viehweg vanishing theorem for surfaces \cite[Theorem 1.3]{tanaka_minimal_2016-1}, that is,
\[
R^i f_* \mathcal{O}_{X_E}(D_E) = 0 \quad \text{for all } i > 0.
\]
This allows us to relate the cohomology of $X_E$ with that of $Y_E$, where $Y_E$ is the birational transform of $X_E$ under $F$. In other words,
\[
H^i (X_E, \mathcal{O}_{X_E}(D_E)) \cong H^i (Y_E, f_* \mathcal{O}_{X_E}(D_E)) \cong H^i (Y_E, \mathcal{O}_{Y_E}(f_* D_E)).
\]
Therefore, it is sufficient to show \[
R^i g_* \mathcal{O}_Y(D_Y) = 0 \quad \text{for all } i > 0
\] near $g(S).$
We intend to apply Lemma \ref{l5} to the excellent plt  pair $(Y, S+B)$ and the proper morphism $g: Y \longrightarrow \mathcal{V}_{\eta}$ to obtain the desired result. In the next lines, we show that the hypotheses of Lemma \ref{l5} are satisfied.  

We know that $D_E - (K_{X_E} + \Delta_E)$ (on $X_E$) is big, this implies $D_Y - (K_Y + \Delta_Y)$ is big. We reduce to the case $D_Y - (K_Y + \Delta_Y )$ is nef and big. 
Indeed, we can replace $Y$ with the output of a $D_Y - (K_Y + \Delta_Y)$-MMP over $\mathcal{V}_{\eta}$, since 
\[
\epsilon(D_Y-  (K_Y+\Delta_Y)) \;\sim_{\mathbb{R}}\; 
K_Y+\Delta_Y+\epsilon D_Y \;-\; (1+\epsilon)(K_Y+\Delta_Y),
\]
where $-(K_Y+\Delta_Y)$ is $g$-ample (see Lemma \ref{l2}), we choose $\epsilon $ small enough so that $\epsilon D_Y \;-\; (1+\epsilon)(K_Y+\Delta_Y)$ is $g$-ample. We choose an effective $g$-ample 
divisor $A' \sim \epsilon D_Y \;-\; (1+\epsilon)(K_Y+\Delta_Y)$ using \cite[Theorem 2.34]{bhatt2022globallyregularvarietiesminimal}. Then 
the pair 
\[
(Y, \Delta_Y+A')
\]
is plt and therefore, \[
(Y, \Delta_Y)
\] is plt by \cite[Corollary 2.35]{kollar_birational_1998}. 
Since $D_Y - (K_Y + \Delta_Y)$ is $g$-nef and $g$-big, we write \[
D_Y \sim_{g,\mathbb{Q}} K_Y + S + \Delta'_Y + L,
\]
where $L$ is $g$-nef and $g$-big. 

We see that all the hypotheses of Lemma \ref{l5} hold true. Indeed, $g|_S$ satisfies K-V vanishing by induction. 
Applying Lemma \ref{l5} to 
\[
g: Y \to \mathcal{V}_{\eta}
\]
yields
\[
R^i g_* \mathcal{O}_Y(D_Y) = 0 \quad \text{for all } i > 0
\]
near $g(S)$, which completes the proof, and hence the result follows.

\end{proof}

The next proof follows on the same lines as \cite[Theorem 3.1c]{bernasconi_kollar_vanishing}.
\begin{proof}[Proof of Corollary \ref{C2}]
The proof follows exactly as in \cite{bernasconi_kollar_vanishing}. The idea is to take a log resolution \( g: Y \to X \) of the pair \( (X, \Delta) \), and run a Minimal Model Program (MMP) starting with the pair \( (Y, \Theta) \), where
\[
\Theta = g_*^{-1} \Delta + \sum_{i=1}^n E_i,
\]
and \( \{ E_i \} \) denotes the exceptional divisors of \( g \). The authors first establish Grauert–Riemenschneider vanishing for the simple normal crossing (snc) pair \( (Y, \Theta) \), and then show that this vanishing is preserved at each step of the MMP.

At every step, the MMP involves either a flip or a divisorial contraction. In the case of a divisorial contraction, suppose that a divisor \( S_i \) is contracted at the \( i \)-th step. There are two possibilities for \( S_i \):
\begin{itemize}
    \item If \( S_i \) is contracted to a point, we apply the Kawamata–Viehweg vanishing theorem for del Pezzo type surfaces (Theorem \ref{T1}) in place of \cite[Theorem 1.1]{arvidson_vanishing}, thereby extending the argument to arbitrary (possibly imperfect) fields.
    \item If \( S_i \) is contracted to a curve, the argument proceeds exactly as in the original proof.
\end{itemize}

In the first case, the original argument in \cite{bernasconi_kollar_vanishing} relies on the assumption that the residue fields of the closed points are perfect of characteristic \( p > 5 \), in order to apply \cite[Theorem 1.1]{arvidson_vanishing}. In our setting, we instead invoke Theorem \ref{T1} to accommodate arbitrary base fields, thereby generalizing the result.

\end{proof}
\begin{proof}[Proof of Corollary \ref{C1}]
   Let $B \subset \lfloor \Delta \rfloor$ be a divisor. We shall show that $(X, B)$ is a rational pair.
In particular, it shows that $X$ has rational singularities. Let $g: Y \longrightarrow (X, \Delta)$ be a thrifty log resolution. We shall show that $g$ is a rational resolution. Let $B_Y$ and $\Delta'_Y$ be the strict transforms of $B$ and $\Delta' = \Delta - B$, respectively.

We write:
$$
K_Y + B_Y + \Delta'_Y + F = g^*(K_X + \Delta) + E
$$
such that $E$ and $F$ are exceptional effective divisors without common exceptional components, and $\lfloor F \rfloor = 0$. We observe that
$$
\lceil E \rceil - B_Y \sim_{g, \mathbb{R}} K_Y + \Delta'_Y + F + (\lceil E \rceil - E).
$$

Using Corollary \ref{C2}, we apply Grauert–Riemenschneider vanishing to $\mathcal{O}_Y(\lceil E \rceil - B_Y)$ and $\omega_Y(B_Y)$ to obtain:
$$
R^i g_* \mathcal{O}_Y(\lceil E \rceil - B_Y) = 0 \quad \text{and} \quad R^i g_* \omega_Y(B_Y) = 0.
$$

The rest of the proof follows along the same lines as \cite[Theorem 2.87]{kollar_singularities_2013}.

Since klt singularities are dlt, and admit a thrifty log resolution, it follows that klt singularities are rational.
\end{proof}

\begin{proof}[Proof of Corollary \ref{C3}]
   The proof is the same as in \cite[Theorem 1.2]{bernasconi2020vanishingtheoremthreefoldscharacteristic}. The key step is to show that \( R^i f_* \mathcal{O}_X = 0 \). It is shown in \cite{tanaka_vanishing_2017} and \cite{bernasconi_tanaka} that this higher direct image sheaf vanishes on an open subset whose complement consists of finitely many points. To handle the behavior at these points, the authors consider a 
   birational transform \( Y \), as constructed in Lemma \ref{l2}. In particular, to prove the vanishing at the point \( v \), it suffices to show that \( R^i g_* \mathcal{O}_Y = 0 \). This follows from the argument used in the proof of the claim appearing in Theorem \ref{T1}, where we apply our generalization to extend the result to the imperfect field setting

\end{proof}

\bibliographystyle{acm}
\bibliography{Library}

\begin{thebibliography}{10}

\bibitem{arvidson_vanishing}
{\sc Arvidsson, E., Bernasconi, F., and Lacini, J.}
\newblock On the {K}awamata--{V}iehweg vanishing theorem for log del {P}ezzo surfaces in positive characteristic.
\newblock {\em Compos. Math. 158}, 4 (2022), 750--763.

\bibitem{baudin2025grauertriemenschneidervanishingcohenmacaulayschemes}
{\sc Baudin, J., Kawakami, T., and Rösler, L.}
\newblock On grauert-riemenschneider vanishing for cohen-macaulay schemes of klt type, 2025.

\bibitem{bernasconi2020vanishingtheoremthreefoldscharacteristic}
{\sc Bernasconi, F.}
\newblock A vanishing theorem for threefolds in characteristic $p>5$ and applications, 2020.

\bibitem{bernasconi_kollar_vanishing}
{\sc Bernasconi, F., and Koll\'{a}r, J.}
\newblock Vanishing theorems for threefolds in characteristic $p>5$.
\newblock {\em Int. Math. Res. Not. 2023}, 4 (2023), 2846--2866.

\bibitem{bernasconi_tanaka}
{\sc Bernasconi, F., and Tanaka, H.}
\newblock On del {P}ezzo fibrations in positive characteristic.
\newblock {\em J. Inst. Math. Jussieu doi:10.1017/S1474748020000067\/} (2020).

\bibitem{bhatt_cohen}
{\sc Bhatt, B.}
\newblock Cohen-{M}acaulayness of absolute integral closures.
\newblock {\em arXiv:2008.08070\/} (2020).

\bibitem{bhatt2022globallyregularvarietiesminimal}
{\sc Bhatt, B., Ma, L., Patakfalvi, Z., Schwede, K., Tucker, K., Waldron, J., and Witaszek, J.}
\newblock Globally +-regular varieties and the minimal model program for threefolds in mixed characteristic, 2022.

\bibitem{bruns_cohen-macaulay_1993}
{\sc Bruns, W., and Herzog, J.}
\newblock {\em Cohen-{Macaulay} rings}, vol.~39 of {\em Cambridge {Studies} in {Advanced} {Mathematics}}.
\newblock Cambridge University Press, Cambridge, 1993.

\bibitem{cascini_log_2017}
{\sc Cascini, P., Tanaka, H., and Witaszek, J.}
\newblock On log del {Pezzo} surfaces in large characteristic.
\newblock {\em Compos. Math. 153}, 04 (2017), 820--850.

\bibitem{das_waldron_imperfect}
{\sc Das, O., and Waldron, J.}
\newblock On the log minimal model program for $3$-folds over imperfect fields of characteristic $p>5$.
\newblock {\em J. Lond. Math. Soc. (2) 106}, 4 (2022), 3895--3937.

\bibitem{MR621766}
{\sc Elkik, R.}
\newblock Rationalit\'e{} des singularit\'es canoniques.
\newblock {\em Invent. Math. 64}, 1 (1981), 1--6.

\bibitem{Filipazzi_2024}
{\sc Filipazzi, S., and Waldron, J.}
\newblock Connectedness principle for 3-folds in characteristic p\&gt;5.
\newblock {\em Michigan Mathematical Journal 74}, 4 (Sept. 2024).

\bibitem{gongyo_rational_2015-1}
{\sc Gongyo, Y., Nakamura, Y., and Tanaka, H.}
\newblock Rational points on log {Fano} threefolds over a finite field.
\newblock {\em J. Eur. Math. Soc. 21}, 12 (2019), 3759--3795.

\bibitem{hacon_witaszek_rationality}
{\sc Hacon, C.~D., and Witaszek, J.}
\newblock On the rationality of {K}awamata log terminal singularities in positive characteristic.
\newblock {\em Algebraic Geometry 6}, 5 (2019), 516--529.

\bibitem{kollar_singularities_1996}
{\sc Koll{\'a}r, J.}
\newblock Singularities of {Pairs}.
\newblock In {\em eprint {arXiv}:alg-geom/9601026\/} (1996), p.~1026.

\bibitem{kollar_local_2011}
{\sc Koll{\'a}r, J.}
\newblock A local version of the {Kawamata}-{Viehweg} vanishing theorem.
\newblock {\em Pure Appl. Math. Q. 7}, 4, Special Issue: In memory of Eckart Viehweg (2011), 1477--1494.

\bibitem{kollar_singularities_2013}
{\sc Koll{\'a}r, J.}
\newblock {\em Singularities of the minimal model program}, vol.~200 of {\em Cambridge {Tracts} in {Mathematics}}.
\newblock Cambridge University Press, Cambridge, 2013.

\bibitem{kollar_birational_1998}
{\sc Koll{\'a}r, J., and Mori, S.}
\newblock {\em Birational geometry of algebraic varieties}, vol.~134 of {\em Cambridge {Tracts} in {Mathematics}}.
\newblock Cambridge University Press, Cambridge, 1998.

\bibitem{lacini_dP}
{\sc Lacini, J.}
\newblock On rank one log del {P}ezzo surfaces in characteristic different from two and three.
\newblock {\em Adv. Math. 442\/} (2024), 109568.

\bibitem{patakfalvi_depth_2014}
{\sc Patakfalvi, Z., and Schwede, K.}
\newblock Depth of \${F}\$-singularities and base change of relative canonical sheaves.
\newblock {\em J. Inst. Math. Jussieu 13}, 1 (2014), 43--63.

\bibitem{stacks-project}
{\sc {Stacks project authors}, T.}
\newblock The stacks project.
\newblock \url{https://stacks.math.columbia.edu}, 2025.

\bibitem{tanaka_vanishing_2017}
{\sc Tanaka, H.}
\newblock Vanishing theorems of {Kodaira} type for {Witt} canonical sheaves.
\newblock {\em arXiv:1707.04036 [math]\/} (2017).

\bibitem{tanaka_minimal_2016-1}
{\sc Tanaka, H.}
\newblock Minimal model program for excellent surfaces.
\newblock {\em Ann. Inst. Fourier 68}, 1 (2018), 345--376.

\bibitem{totaro_failure}
{\sc Totaro, B.}
\newblock The failure of {K}odaira vanishing for {F}ano varieties, and terminal singularities that are not {C}ohen-{M}acaulay.
\newblock {\em J. Algebraic Geom. 28\/} (2019), 751--771.

\bibitem{totaro2024terminal3foldscohenmacaulay}
{\sc Totaro, B.}
\newblock Terminal 3-folds that are not cohen-macaulay, 2024.

\end{thebibliography}

\end{document}